\documentclass[12pt]{article}
\usepackage{amssymb}
\usepackage{amsmath}
\usepackage{graphicx}
\newtheorem{theorem}{Theorem}[section]

\newtheorem{corollary}[theorem]{Corollary}

\newtheorem{definition}[theorem]{Definition}
\newtheorem{example}[theorem]{Example}

\newtheorem{note}[theorem]{Note}
\newtheorem{proposition}[theorem]{Proposition}

\newenvironment{proof}[1][Proof]{\noindent\textbf{#1.} }{\ \rule{0.5em}{0.5em}}
\begin{document}
	\begin{center}
		\large
			Influence of ideals in compactifications\\[1ex]
		\large
		Manoranjan Singha$^*$, Sima Roy$^{**}$\\[2ex]
		\scriptsize
		\textsuperscript{*}\emph{Department of Mathematics, University of North Bengal, Darjeeling-734013, India}
		
		\textsuperscript{**}\emph{Department of Mathematics, Raja Rammohun Roy Mahavidyalaya, Hooghly-712406, India}
		
	\end{center}
	
	\date{}

\hrule
\vspace{0.2cm}
\begin{flushleft}
	\textbf{ABSTRACT}
\end{flushleft}
	One point compactification is studied in the light of ideal of subsets of $\mathbb{N}$. $\mathcal{I}$-proper map is introduced and showed that a continuous map can be extended continuously to the one point $\mathcal{I}$-compactification if and only if the map is $\mathcal{I}$-proper. Shrinking condition(C) introduced in this article plays an important role to study various properties of $\mathcal{I}$-proper maps. It is seen that one point $\mathcal{I}$-compactification of a topological space may fail to be Hausdorff but a class $\{\mathcal{I}\}$ of ideals has been identified for which one point $\mathcal{I}$-compactification coincides with the one point compactification if it is metrizable.\\


\vspace{0.2cm}
\hrule
\footnotetext{
	$^*$E-mail: manoranjan$\_$singha@rediffmail.com
	\vspace{0.1cm}
	
	$^{**}$E-mail: rs$\_$sima@nbu.ac.in
}

\section{Introduction}
Theory of statistical convergence gets birth in the year 1951 as an extension of the concept of convergence of sequence of real numbers \cite{Fast}. But it gets acquaintance in the late twentieth century. It has huge applications in the theory of integrability and related summability methods \cite{Schoenberg}, \cite{Kocinac1}. Thin \cite{Fridy} subsets, which form ideal \cite{Kuratowski} on $\mathbb{N}$, plays main role in statistical convergence. After a long, in the year 2001,  Kostyrko et al \cite{Kostyrkoa} introduced the concept of $\mathcal{I}$-convergence using the notion of ideals. During last two decades research on the theory of $\mathcal{I}$-convergence have been reached in the peak \cite{Sever}, \cite{Savas}, \cite{Kocinac}, \cite{Cincura}, \cite{Das2}, \cite{Boccuto} etc..

Let's begin with some basic definitions and results.

For any non-empty set $X$, a family $\mathcal{I}\subset 2^X$ is called an ideal if $(1)$ $\emptyset \in \mathcal{I}$, $(2)$ $A,B\in \mathcal{I}$ implies $A\cup B\in \mathcal{I}$, and $(3)$ $A\in \mathcal{I},B\subset A$ implies $B\in \mathcal{I}$ \cite{Kuratowski}. An ideal $\mathcal{I}$ is called non-trivial if $\mathcal{I}\neq \{\emptyset\}$ and $X\notin \mathcal{I}.$ A non-trivial ideal $\mathcal{I}\subset 2^X$ is called admissible if it contains all the singleton sets \cite{Kuratowski}. Various examples of non-trivial admissible ideals are given in \cite{Kostyrkoa}.

A sequence $(x_n)_{n\in\mathbb{N}}$ in a metric space ${\rm(X,d)}$ is said to be $\mathcal{I}$-convergent to $\xi\in X$ $(\xi=\mathcal{I}-\lim_{n\to\infty} x_n)$ if and only if for each $\epsilon >0$ the set $A(\epsilon)=\{n\in\mathbb{N}: d(x_n,\xi)\geq \epsilon\}$ belongs to $\mathcal{I}$. The element $\xi$ is called the $\mathcal{I}$-limit of the sequence $(x_n)_{n\in\mathbb{N}}$ \cite{Kostyrkoa}. Throughout this article $\mathcal{I}$ is a nontrivial admissible ideal on $\mathbb{N},$ unless otherwise stated.

A sequence $x=(x_n)_{n\in M}$ in a topological space $X$ is called $\mathcal{I}$-thin, where $\mathcal{I}$ is a nontrivial admissible ideal on $\mathbb{N}$ if $M\in \mathcal{I}$ ; otherwise it is called $\mathcal{I}$-nonthin \cite{Singha}.

A topological space $X$ is $\mathcal{I}$-compact if any $\mathcal{I}$-nonthin sequence $(x_n)_{n\in K}$ in $X$ has an $\mathcal{I}$-nonthin subsequence $(x_n)_{n\in M}$ that $\mathcal{I}/_M$-converges to some point in $X$ \cite{Singha}. And then several properties of $\mathcal{I}$-compactness have been studied.

\section{$\mathcal{I}$-proper maps and $\mathcal{I}$-compactification}

\begin{definition} An $\mathcal{I}$-nonthin sequence $(x_n)_{n\in M}$ in a topological space $X$ is $\mathcal{I}$-eventually constant at $\alpha$ if $\{n\in M : x_n\neq\alpha\}\in \mathcal{I}/_M$.
\end{definition}
\begin{note} Every eventually constant sequence is $I$-eventually constant, but converse may not be true.
	For example, consider $\mathcal{I}=P(2\mathbb{N})\cup \mathcal{I}_f$, $\mathcal{I}_f$ is the collection of all finite subsets of $\mathbb{N}$ and $x_n=0,$ if n is odd and $x_n=1,$ if n is even.
\end{note}
\begin{definition} An $\mathcal{I}$-nonthin sequence $(x_n)_{n\in M}$ in a topological space $X$ is said to be $\mathcal{I}$-eventually in $S\subset X$ if $\{n\in M; x_n\notin S\}\in \mathcal{I}/_M$.
\end{definition}

\begin{definition}(\rm\cite{Singha}) A subset $A$ of a topological space $X$ is called $\mathcal{I}$-closed if $A=\overline{A}^\mathcal{I}$, where $\overline{A}^\mathcal{I}=\{x\in X;$ there exists an $\mathcal{I}$-nonthin sequence $(x_n)_{n\in L}$ in $A$ that $\mathcal{I}/_L$-converges to $x\}$. A subset $A$ of $X$ is $\mathcal{I}$-open if complement of $A$ is $\mathcal{I}$-closed that is, if an $\mathcal{I}$-nonthin sequence $(x_n)_{n\in M}$ $\mathcal{I}/_M$-converges to some point in $A$, then $(x_n)_{n\in M}$ is $\mathcal{I}$-eventually in $A$.
\end{definition}

\begin{definition}
	An admissible ideal $\mathcal{I}$ is said to satisfy shrinking condition(C) if for any set $A\notin \mathcal{I}$ there exists a subset $B$ of $A$ such that $B\notin \mathcal{I}$  and no infinite subset of $B$ is in $\mathcal{I}$.
\end{definition}

The following examples are an witness of such ideal.

\begin{example}
	\begin{enumerate}
		\item Consider the ideal $\mathcal{I}_1=P(2\mathbb{N})\cup\mathcal{I}_f$, where $\mathcal{I}_f$ is the set of all finite subsets of $\mathbb{N}$.
		\item Let $\mathcal{I}_2=\{A; A$ intersects finite $\Delta_i's\}$ and $\mathcal{I}_3=\{A; A\cap\Delta_i=$ finite, for all $i\in\mathbb{N}\}$  where $\mathbb{N}=\underset{i\in\mathbb{N}}\cup \Delta_i$, each $\Delta_i$ is infinite and $\Delta_i\cap\Delta_j=\phi,$ for all $i\neq j$.
	\end{enumerate}
	Then $\mathcal{I}_1,\mathcal{I}_2,\mathcal{I}_3$ satisfy shrinking condition(C).
\end{example}

\begin{proposition}\label{natural}
	Let $\mathcal{I}$ satisfies shrinking condition(C). If $(n_k)_{k\in \mathbb{N}}$ be any sequence of natural numbers with $\underset{k\rightarrow\infty}\lim {n_k}=\infty$ and the range set of $(n_k)_{k\in \mathbb{N}}$ is not in $\mathcal{I}$, then there exists a monotone strictly increasing subsequence of $(n_k)_{k\in \mathbb{N}}$ whose range set also not in $\mathcal{I}$.
\end{proposition}
The ideal $\mathcal{I}_d$ consisting of subsets of $\mathbb{N}$ having natural density $0$ does not satisfy shrinking condition(C) due to Proposition \ref{natural}, since if we take a sequence $(x_n)_{n\in \mathbb{N}}$ where $x_1=2, x_2=1$ and $x_n=2^{k+1}-(r-1)$, if $n=2^k+r$, $1\leq r\leq 2^k$, $k\in \mathbb{N}$, there is no monotone strictly increasing subsequence  whose range set not in $\mathcal{I}_d$.

\begin{theorem}\label{Iclosed}
	Let $X$ be $T_1$ and $\mathcal{I}$ satisfies shrinking condition$(C)$. Let, $(x_n)_{n\in L}$ be an $\mathcal{I}$-nonthin sequence in $X$ having no $\mathcal{I}$-nonthin subsequence $(x_n)_{n\in K}$ that $\mathcal{I}/_K$-convergent. Then there exists a subset $P\notin \mathcal{I}$ such that the set $\{(x(n),n);$ $n\in P\}$ is $\mathcal{I}$-closed in $X\times\mathbb{N}^+$, where $\mathbb{N}^+$ is the one point compactification of $\mathbb{N}$.
\end{theorem}
\begin{proof}
	Since $\mathcal{I}$ satisfies shrinking condition(C), there exists $P\subset L$ such that $P\notin \mathcal{I}$ and no infinite subset of $P$ is in $\mathcal{I}$. Therefore $(x_n)_{n\in P}$ has no convergent subsequence and from Proposition 1.3 in \cite{Brown}, the set $\{(x(n),n);$ $n\in P\}$ is sequentially closed in $X\times \mathbb{N}^+$. Henceforth  the set $\{(x(n),n);$ $n\in P\}$ is $\mathcal{I}$-closed in $X\times\mathbb{N}^+$.
\end{proof}

\begin{definition}
	Let $X$ and $Y$ be topological spaces and $f: X\rightarrow Y$ be a function.
	
	\begin{enumerate}
		\item $f$ is $\mathcal{I}$-continuous if for any $\mathcal{I}$-nonthin sequence $(x_n)_{n\in M}$ which is $\mathcal{I}/_M$-converges to $x$, then $(f(x_n))_{n\in M}$ is $\mathcal{I}/_M$-converges to $f(x)$.
		\item $f$ is $\mathcal{I}$-homeomorphism if $f$ is bijective, $\mathcal{I}$-continuous and $f^{-1}$ is $\mathcal{I}$-continuous.
		\item $f$ is an $\mathcal{I}$-embedding if $f$ yields an $\mathcal{I}$-homeomorphism between $X$ and $f(X)$.
		\item $f$  is $\mathcal{I}$-closed if image of any $\mathcal{I}$-closed set is $\mathcal{I}$-closed.
		\item $f$ is $\mathcal{I}$-proper if $f\times 1_z: X\times Z\rightarrow Y\times Z$ is $\mathcal{I}$-closed, for all spaces $Z$, provided $f$ is $\mathcal{I}$-continuous.
	\end{enumerate}
\end{definition}

\begin{note}
	Every $\mathcal{I}$-proper function is $\mathcal{I}$-closed.
\end{note}
\begin{theorem}\label{closed} Let $X$ and $Y$ be topological spaces. A function $f: X\rightarrow Y$ is $\mathcal{I}$-continuous if and only if $f^{-1}(B)$ is $\mathcal{I}$-closed for every $\mathcal{I}$-closed subset $B$ of $Y$.
\end{theorem}
\begin{proof} Proof is omitted.
\end{proof}

\begin{definition} A topological space is called an $\mathcal{I}$-US space if every $\mathcal{I}$-nonthin $\mathcal{I}/_M$-convergent sequence $(x_n)_{n\in M}$ has exactly one $\mathcal{I}$-limit to which it converges.
\end{definition}

\begin{theorem}\label{proper}
	Let $X$ and $Y$ be topological spaces. Let $f: X\rightarrow Y$ be $\mathcal{I}$-continuous and $Y$ is $\mathcal{I}$-US. Consider the following conditions:
	\begin{description}
		\item[(a)] If an $\mathcal{I}$-nonthin sequence $(x_n)_{n\in M}$ in $X$ has no $\mathcal{I}$-nonthin subsequence $(x_n)_{n\in N}$ that  $\mathcal{I}/_N$-convergent in $X$, then $(f(x_n))_{n\in M}$ has no $\mathcal{I}$-nonthin subsequence $(x_n)_{n\in L}$ that $\mathcal{I}/_L$-convergent in $Y$.
		\item[(b)]  $f^{-1}(B)$ is  $\mathcal{I}$-compact for every $\mathcal{I}$-compact subset $B$ of $Y$.
		\item[(c)] If $(x_n)_{n\in M}$ is an $\mathcal{I}$-nonthin $\mathcal{I}/_M$-convergent sequence, then $f^{-1}(\bar{x})$ is $\mathcal{I}$-compact, where $\bar{x}$ is the union of $(x_n)_{n\in M}$ and its $\mathcal{I}$-limit.
		\item[(d)] $f$ is $\mathcal{I}$-proper.
		\item[(e)] $f\times 1: X\times\mathbb{N}^+\rightarrow Y\times\mathbb{N}^+$ is $\mathcal{I}$-closed, where $\mathbb{N}^+$ is the one point compactification of $\mathbb{N}$.
		
		Then, $(a)\Rightarrow (b)\Rightarrow (c)\Rightarrow(d) \Rightarrow(e)$ and if $X$ is $T_1$ and $\mathcal{I}$ satisfies shrinking condition $(C)$, $(e)\Rightarrow (a)$.
	\end{description}
\end{theorem}
\begin{proof}
	$(a)\Rightarrow (b)$ Let $(x_n)_{n\in M}$ be an $\mathcal{I}$-nonthin sequence in $f^{-1}(B)$. If $(x_n)_{n\in M}$ has no $\mathcal{I}$-nonthin subsequence $(x_n)_{n\in L}$ that $\mathcal{I}/_L$-convergent in $X$, then $(f(x_n))_{n\in M}$ has no  $\mathcal{I}$-nonthin subsequence $(f(x_n))_{n\in L}$ that $\mathcal{I}/_L$-convergent in $Y$. Since $f(x_n)\in B$, which contradicts $\mathcal{I}$-compactness of $B$. So, $(x_n)_{n\in M}$ has an $\mathcal{I}$-nonthin subsequence $(x_n)_{n\in L}$ that $\mathcal{I}/_L$-converges to $l$(say). As $f$ is $\mathcal{I}$-continuous, $(f(x_n))_{n\in L}$ is $\mathcal{I}/_L$-convergent to $f(l)$. Also since $B$ is $\mathcal{I}$-compact and $Y$ is $\mathcal{I}$-US, $B$ is $\mathcal{I}$-closed. So $f(l)\in B$ which implies $l \in f^{-1}(B)$. Hence  $f^{-1}(B)$ is  $\mathcal{I}$-compact.
	
	$(b)\Rightarrow (c)$ and  $(d)\Rightarrow (e)$ are trivially hold.
	
	$(c)\Rightarrow (d)$ Let $Z$ be a topological space and $A$ is $\mathcal{I}$-closed in $X\times Z$. Also let $(y_n,z_n)_{n\in M}$ be an $\mathcal{I}$-nonthin sequence in $(f\times 1_Z)(A)$  that $\mathcal{I}/_M$-converges to $(y,z)$. So there exists a sequence $(x_n)_{n\in M}$ in $X$ such that $f(x_n)=y_n$. Then $(x_n,z_n)_{n\in M}$ is an $\mathcal{I}$-nonthin sequence in $A$ and $(x_n)_{n\in M}$ is an $\mathcal{I}$-nonthin sequence in $f^{-1}(\bar{y})$, where $\bar{y}$ is the union of $(y_n)_{n\in M}$ and its $\mathcal{I}$-limit. Therefore $(x_n)_{n\in M}$ has an $\mathcal{I}$-nonthin subsequence $(x_n)_{n\in L}$ that $\mathcal{I}/_L$-convergent to some point $l$(say). So, $(x_n,z_n)_{n\in L}\rightarrow_{\mathcal{I}/_L}(l,z)$. Since $A$ is $\mathcal{I}$-closed in $X\times Z$, $(l,z)\in A$. Henceforth, $(f(x_n))_{n\in M}$ is $\mathcal{I}/_L$ convergent to $f(l)$ and $y$ also. Since $Y$ is $\mathcal{I}$-US, $f(l)=y$ and so $(y,z)\in(f\times 1_Z)(A)$.
	
	$(e)\Rightarrow (a)$ This implication is immediate from Theorem \ref{Iclosed}.
\end{proof}

Since every locally compact Hausdorff space can be embedded into a compact Hausdorff space, likewise using the notion of ideals of subsets of $\mathbb{N}$, following theorem ensures that every topological space can be $\mathcal{I}$-embedded into an $\mathcal{I}$-compact space.

\begin{theorem}\label{newtop}
	Let $X$ be a topological space, then $X$ can be  $\mathcal{I}$-embedded into an $\mathcal{I}$-compact space $\widehat{X}$ so that $\widehat{X}-X$ contains exactly one point and $X$ is an open dense subspace of $\widehat{X}$.
\end{theorem}

\begin{proof}
	Let's consider a topology on $\widehat{X}$, $\widehat{\tau}$=$\tau\cup\{U\cup\{\alpha\};$ $X-U$ is closed and $\mathcal{I}$-compact in $X\}$=$\tau\cup\{U\cup\{\alpha\};$ $U$ is open in $X$ and any $\mathcal{I}$-nonthin sequence $(x_n)_{n\in L}$ in $X$ having no $\mathcal{I}$-nonthin $\mathcal{I}/_M$-convergent subsequence $(x_n)_{n\in M}$ is $\mathcal{I}$-eventually in $U$$\}$, where $\alpha$ is the point at infinity of $X$. Then $X$ is an open dense subspace of $\widehat{X}$. Let $S(X)$ be the set of all $\mathcal{I}$-nonthin sequence $(x_n)_{n\in M}$ in $X$ having no $\mathcal{I}$-nonthin  $\mathcal{I}/_L$-convergent subsequence $(x_n)_{n\in L}$. Let $U\cup\{\alpha\}$ be any open set containing $\alpha$, then all the elements of $S(X)$ is $\mathcal{I}$-eventually in $U$. Therefore all the elements of $S(X)$ is $\mathcal{I}$-converges to $\alpha$. Hence $\widehat{X}$ is $\mathcal{I}$-compact.
\end{proof}
\begin{definition} In the Theorem \ref{newtop} $\widehat{X}$ is known as one point $\mathcal{I}$-compactification of $X$.
\end{definition}

\begin{definition} A topological space is said to be $\mathcal{I}$-sequential if every $\mathcal{I}$-closed set is closed.
\end{definition}

\begin{theorem}
	If a topological space $X$ is $\mathcal{I}$ sequential, then $\widehat{X}$ is $\mathcal{I}$-sequential. In addition if $X$ is $\mathcal{I}$-US, then $\widehat{X}$ is $\mathcal{I}$-US.
\end{theorem}
\begin{proof}
	Let, $U$ be an $\mathcal{I}$-open subset of $\widehat{X}$. Then, $U\cap X$ is $\mathcal{I}$-open in $X$ and since $X$ is $\mathcal{I}$-sequential, $U\cap X$ is open in $X$. If $\alpha\notin U$, then $U$ is open in $X$. Let $\alpha\in U$. Since $U$ is $\mathcal{I}$-open and all the elements of $S(X)$ is $\mathcal{I}$-converges to $\alpha$, then every element of $S(X)$ is $\mathcal{I}$-eventually in $U$. Also $U-\{\alpha\}=U\cap X$ is open in $X$. Hence $U$ is open in $\widehat{X}$ and $\widehat{X}$ is $\mathcal{I}$-sequential.
	
	Now let us consider an $\mathcal{I}$-nonthin sequence $(x_n)_{n\in M}$ in $\widehat{X}$ that is $\mathcal{I}/_M$-converges to $x$ and $y$ in $\widehat{X}$. If both $x,y\in X$, then $x=y$. Let, $(x_n)_{n\in M}$ $\mathcal{I}/_M$-converges to $x\in X$. Since $X$ is open in $\widehat{X}$, $\{n\in M; x_n\in X\}=L\notin \mathcal{I}/_M$. That is, $(x_n)_{n\in L}$ is an $\mathcal{I}$-nonthin sequence in $X$. Since $X$ is $\mathcal{I}$-US, $A=\{x\}\cup\{x_n; n\in L\}$ is $\mathcal{I}$-closed in $X$. So, $A$ is closed in $X$ that is $X-A=U$ is open in $X$. Therefore all the elements of $S(X)$ is $\mathcal{I}$-eventually in $U$, since $A$ is $\mathcal{I}$-compact. So, $U\cup\{\alpha\}$ is open in $\widehat{X}$ which implies $(x_n)_{n\in L}$ does not converge to $\alpha$. Hence $\widehat{X}$ is $\mathcal{I}$-US.
\end{proof}
\begin{theorem}\label{continuity}
	Let $X$, $Y$ be topological spaces and $Y$ be an $\mathcal{I}$-US space containing $X$ as an open subspace. Let, $f: Y\rightarrow \widehat{X}$ is defined by
	$$f(x)=\left\{
	\begin{array}{ll}
		x & \hbox{,~if~$x\in X;$} \\
		\alpha & \hbox{,~if~$x\notin X$}
	\end{array}
	\right.$$
	$\alpha$ is the point at infinity of $X$.
	Then, $f$ is $\mathcal{I}$-continuous.
\end{theorem}
\begin{proof}
	Let, $(x_n)_{n\in M}$ be an $\mathcal{I}$-nonthin sequence in $Y$ which is $\mathcal{I}/_M$-converges to $y$. If $y\in X$, then $(x_n)_{n\in M}$ is $\mathcal{I}$-eventually in $X$ (since $X$ is an open subspace of $Y)$. Then, $(f(x_n))_{n\in M}$ is $\mathcal{I}$-eventually in $X$ and $\mathcal{I}/_M$-converges to $f(y)$.
	If $y\in Y-X$, then three cases may arise. First case, if $(x_n)_{n\in M}$ has only $\mathcal{I}$-thin subsequences in $X$, then $(f(x_n))_{n\in M}$ is $\mathcal{I}$-eventually at $\alpha$ and so $f(x_n)\rightarrow_{\mathcal{I}/_M} \alpha$. For second case, let $(x_n)_{n\in M}$ has only $\mathcal{I}$-thin subsequences in $Y-X$. Since $Y$ is $\mathcal{I}$-US and $y\in Y-X$, which implies $(x_n)_{n\in M}$ has no $\mathcal{I}$-nonthin subsequence $(x_n)_{n\in L}$ $\mathcal{I}/_L$-convergent in $X$. But $\widehat{X}$ is $\mathcal{I}$-compact, so $f(x_n)\rightarrow_{\mathcal{I}/_M} \alpha$. Finally, if $(x_n)_{n\in M}$ has $\mathcal{I}$-nonthin subsequence in both $X$ and $Y-X$, then using above two cases both subsequences $\mathcal{I}$-converges to $\alpha$ and so, $f(x_n)\rightarrow_{\mathcal{I}/_M} \alpha$. Hence $f$ is $\mathcal{I}$-continuous.
\end{proof}

\begin{theorem} Let $Y$ be an $\mathcal{I}$-compact $\mathcal{I}$-US space containing  $\mathcal{I}$-sequential space $X$ as an open dense subspace and $Y-X$ has exactly one point. Then there is an $\mathcal{I}$-homeomorphism of $Y$ with $\widehat{X}$ which is the identity on $X$.
\end{theorem}
\begin{proof}
	The proof of the theorem follows from Theorem \ref{closed} and Theorem \ref{continuity}.
\end{proof}

In the following we have investigated the relation between one point $\mathcal{I}$-compactification and $\mathcal{I}$-proper maps. Here $\widehat{f}: \widehat{X}\rightarrow \widehat{Y}$ is the extension of $f:X\rightarrow Y$ which takes  $\alpha_x$, the point at infinity of $X$ to  $\alpha_y$, the point at infinity of $Y$.
\begin{theorem}
	Let $X, Y$ be $\mathcal{I}$-sequential, $\mathcal{I}$-US spaces and $f:X\rightarrow Y$ be continuous. Let $X$ be $T_1$ and $\mathcal{I}$ satisfies shrinking condition $(C)$. Then $\widehat{f}: \widehat{X}\rightarrow \widehat{Y}$ is continuous if and only if $f$ is $\mathcal{I}$-proper.
\end{theorem}

\begin{proof}
	Let $f$ be $\mathcal{I}$-proper and $\alpha_x$ and $\alpha_y$ be the point at infinity of $X$ and $Y$ respectively.
	Let $(x_n)_{n\in M}$ be a sequence in $\widehat{X}$ that $\mathcal{I}/_M$ converging to $\alpha_x$. If $(x_n)_{n\in M}$ has an $\mathcal{I}$-nonthin subsequence which takes the value $\alpha_x$, then $(\widehat{f}(x_n))_{n\in M}$ also has an $\mathcal{I}$-nonthin subsequence which takes the value $\alpha_y$. Thus  $(\widehat{f}(x_n))_{n\in M}$ $\mathcal{I}/_M$-converges to $\alpha_y$. Now if $(x_n)_{n\in M}$ has only $\mathcal{I}$-thin subsequence say $(x_n)_{n\in L}$  which takes the value $\alpha_x$, then $(x_n)_{n\in P}$ is an $\mathcal{I}$-nonthin sequence in $X$ with no $\mathcal{I}$-nonthin subsequence $(x_n)_{n\in {P_1}}$ that $\mathcal{I}/_{P_1}$ convergent in $X$, where $P=M-L$. Then by Theorem \ref{proper}, $(f(x_n))_{n\in P}$ has no  $\mathcal{I}$-nonthin subsequence $(f(x_n))_{n\in {P_1}}$ that $\mathcal{I}/_{P_1}$ convergent in $Y$ and so  $(\widehat{f}(x_n))_{n\in M}$ $\mathcal{I}/_M$-converges to $\alpha_y$. Proof of the converse part follows from Theorem \ref{proper}.
\end{proof}

\begin{definition} A topological space $X$ is locally $\mathcal{I}$-compact if every point $x\in X$ has an $\mathcal{I}$-compact neighbourhood that is for every $x\in X$, there exists an $\mathcal{I}$-compact subset $C$ of $X$ and an open set $U$ containing $x$ such that $U\subset C$.
\end{definition}

\begin{example}
	If we take $\mathcal{I}=\mathcal{I}_f\cup\{A\subset \mathbb{N};$ $A\cap\Delta_i$ is infinite for finite $i's$ and for other $i's$ $A\cap\Delta_i$ is finite$\}$,  where $\mathbb{N}=\underset{i\in \mathbb{N}}\cup\Delta_i$, $\Delta_i\cap\Delta_j=\phi,$ $i\neq j$ and $\mathcal{I}_f$ is the collection of all finite subsets of $\mathbb{N}$. Then $\mathbb{R}$ with usual topology is locally $\mathcal{I}$-compact
	but not $\mathcal{I}$-compact (\cite{Singha}).
\end{example}
\begin{example}
	$\mathbb{R}$ with usual topology is not locally $\mathcal{I}_d$-compact.
\end{example}

\begin{theorem}\label{locIcom} A topological space $X$ is locally $\mathcal{I}$-compact, Hausdorff and $\mathcal{I}$-sequential, then there exists a unique upto homeomorphic topological space $\widehat{X}$ which is $\mathcal{I}$-compact, Hausdorff, $X$ is an open dense subspace of $\widehat{X}$ and $\widehat{X}-X$ contains exactly one element.
\end{theorem}
\begin{proof} Consider the previous mentioned topology $\widehat{\tau}$ on $\widehat{X}$ (in Theorem \ref{newtop}). Let, $x$ and $y$ be two distinct points of $\widehat{X}$. If both $x,y$ in $X$, there exists two disjoint open sets $U$ and $V$ in $X$ containing $x$ and $y$ respectively. Now, let $x\in X$ and $y=\alpha$. Since $X$ is locally $\mathcal{I}$-compact, there exists an $\mathcal{I}$-compact set $C$ in $X$ containing a neighbourhood $U$ of $x$. Also since $\mathcal{I}$-compact subset of a Hausdorff space is $\mathcal{I}$-closed and $X$ is $\mathcal{I}$-sequential, which implies $C$ is closed in $X$. So, $(X-C)\cup\{\alpha\}$ is open in $\widehat{X}$. Henceforth, there exists two disjoint open sets $U$ and $(X-C)\cup\{\alpha\}$ containing $x$ and $\alpha$ respectively. So, $\widehat{X}$ is Hausdorff.
	Let, $\widehat{Y}$ be a Hausdorff one point $\mathcal{I}$-compactification of $X$. Let $\alpha_x$ and $\alpha_y$ be the point at infinity of $X$ and $Y$ respectively and
	define a map $h: \widehat{X}\rightarrow\widehat{Y}$ by
	$$f(x)=\left\{
	\begin{array}{ll}
		x & \hbox{,~if~$x\in X;$} \\
		\alpha_y  & \hbox{,~if~$x=\alpha_x.$}
	\end{array}
	\right.$$
	Then $h$ is a bijection. Let, $U$ be an open set in $\widehat{X}$ not containing $\alpha_x$, then $U\cap X=U$ is open in $X$. Therefore $h(U)=U$ is open in $X$ so in $\widehat{Y}$. Now let $U$ be an open set in $\widehat{X}$ containing $\alpha_x$. Then, $\widehat{X}-U=C$ is closed in $\widehat{X}$ and since $\widehat{X}$ is $\mathcal{I}$-compact, $C$ is $\mathcal{I}$-compact in $\widehat{X}$. Then $C\subset X$ is $\mathcal{I}$-compact in $X$ and so $\mathcal{I}$-compact in $\widehat{Y}$. Also since $X$ is $\mathcal{I}$-sequential, $C$ is closed in $\widehat{Y}$. Therefore $h(U)=\widehat{Y}-C$ is open in $\widehat{Y}$. Hence $h$ is a homeomorphism.
\end{proof}

\begin{theorem} If one point $\mathcal{I}$-compactification of a topological space $X$ is Hausdorff and $\mathcal{I}$-sequential, then $X$ is locally $\mathcal{I}$-compact.
\end{theorem}
\begin{proof}
	Let, $x\in X$ and since $\widehat{X}$ is Hausdorff, there exists two disjoint open sets $U$ and $V$ of $\widehat{X}$ containing $x$ and $\alpha$ respectively. So, $\widehat{X}-V=C$ is closed in $\widehat{X}$ and then $\mathcal{I}$-compact in $\widehat{X}$. Therefore $C$ is $\mathcal{I}$-compact in $X$ and $x\in U\subset C$. Hence $X$ is locally $\mathcal{I}$-compact.
\end{proof}

\begin{corollary} A topological space $X$ is locally $\mathcal{I}$-compact, Hausdorff and $\mathcal{I}$-sequential if and only if there exists a unique upto homeomorphic topological space $\widehat{X}$ which is $\mathcal{I}$ -compact, Hausdorff, $X$ is an open dense subspace of $\widehat{X}$ and $\widehat{X}-X$ has exactly one point.
\end{corollary}

\begin{theorem} If $f: X_1\rightarrow X_2$ is a homeomorphism of locally $\mathcal{I}$-compact, Hausdorff and $\mathcal{I}$-sequential spaces, then $f$ extends to a homeomorphism of their  one point $\mathcal{I}$-compactifications.
\end{theorem}
\begin{proof}
	Let, $f: X_1\rightarrow X_2$ be a homeomorphism, where $X_1$ and $X_2$ are locally $\mathcal{I}$-compact, Hausdroff $\mathcal{I}$-sequential spaces. Then from Theorem \ref{locIcom}, there exist Hausdorff one point $\mathcal{I}$-compactifications $\widehat{X}_1$ and $\widehat{X}_1$ of $X_1$ and $X_2$ respectively.  Define $\widehat{f}: \widehat{X}_1\rightarrow\widehat{X}_2$ by
	$$\widehat{f}(\alpha)=\left\{
	\begin{array}{ll}
		f(x) & \hbox{,~if~$x\in X_1;$} \\
		\alpha_2 & \hbox{,~if~$x=\alpha_1$}
	\end{array}
	\right.$$
	$\alpha_1$ and $\alpha_2$ be the point at infinity of $X_1$ and $X_2$ respectively.
	
	Then $\widehat{f}$ is a homeomorphism.
\end{proof}

\begin{theorem}\label{compactifi1}
	If a topological space $X$ has a Hausdorff one point $\mathcal{I}$-compactification, then every compact subset of $X$ is $\mathcal{I}$-compact.
\end{theorem}
\begin{proof}
	Let, $(X,\tau)$ be a topological space, it has a Hausdorff one point $\mathcal{I}$-compactification. Let, $K\subset X$ is a compact subset of $X$ and $e: X\rightarrow \widehat{X}$ be an embedding. Then, $e(K)$ is compact in $\widehat{X}$. Since $\widehat{X}$ is $T_2$, $e(K)$ is closed in $\widehat{X}$. This implies $e(K)\subset e(X)$ is $\mathcal{I}$-compact in $\widehat{X}$ and so $K$ is $\mathcal{I}$-compact.
\end{proof}

\begin{example}\label{noHausdorff}
	Since not every compact subset of $\mathbb{R}$ is $\mathcal{I}_d$-compact, so one point $\mathcal{I}_d$-compactification of $\mathbb{R}$ is not Hausdporff.
\end{example}

\begin{definition}(\cite{Singha})
	An admissible ideal $\mathcal{I}$ is said to satisfy shrinking condition{\rm(B)} if for any sequence of sets $\{A_i\}$ not in $\mathcal{I}$, there exists a sequence of sets $\{B_i\}$ in $\mathcal{I}$ such that each $B_i\subset A_i$  and $B={\overset{\infty}{\underset{i=1}\cup}}B_i\notin \mathcal{I}.$
\end{definition}

\begin{theorem}
	If $\mathcal{I}$ satisfies shrinking condition(B) and one point compactification of a topological space is metrizable, then one point $\mathcal{I}$-compactification is homeomorphic to one point compactification.
\end{theorem}
\begin{proof}
	Let $\widehat{X}$ and $Y$ be the one-point $\mathcal{I}$-compactification of $X$ and one point compactification of $X$ respectively. Since $\mathcal{I}$-satisfies shrinking condition(B) and $Y$ is metrizable, this implies $Y$ is  $\mathcal{I}$-compact \cite{Singha}. Claim that, $Y$ and $\widehat{X}$ are homeomorphic. Let $\alpha_1$ and $\alpha_2$ be the point at infinity of $X$ with $\widehat{X}-X=\{\alpha_1\}$ and $Y-X=\{\alpha_2\}$. Define a mapping $h: \widehat{X}\rightarrow Y$ by
	$$h(x)=\left\{
	\begin{array}{ll}
		x & \hbox{,~if~$x\in X;$} \\
		\alpha_2 & \hbox{,~if~$x=\alpha_1.$}
	\end{array}
	\right.$$
	If $U$ be an open set in $\widehat{X}$ not containing $\alpha_1$, then $U$ is open in $Y$ also. Now let, $U$ is an open set in $\widehat{X}$ containing $\alpha_1$, then $C=\widehat{X}-U$ is closed in $\widehat{X}$. Since $\widehat{X}$ is $\mathcal{I}$-compact, $C$ is $\mathcal{I}$-compact in $\widehat{X}$. Also metrizability of $X$ implies $C$ is compact in $X$. Then $C$ is compact and hence closed in $Y$. Therefore, $h(U)=Y-C$ is open in $Y$ and so $h^{-1}$ is continuous.
	Now let $U$ be an open set in $Y$ containing $\alpha_2$, that is $Y-U=C$ is closed in $Y$. Then $C$ is compact in $Y$ and so in $X$. Also since $X$ is $T_2$, $C$ is closed in $X$ and since $\mathcal{I}$-satisfies shrinking condition(B), $C$ is $\mathcal{I}$-compact in $X$. Therefore, $(X-C)\cup\{\alpha_1\}$ is open in $\widehat{X}$. So $h^{-1}(U)=\widehat{X}-C$ is open in $\widehat{X}$ that is, $h$ is continuous.
\end{proof}

\begin{corollary}
	
	If $\mathcal{I}$ satisfies shrinking condition(B), then one point $\mathcal{I}$ compactification of $\mathbb{R}$ with usual topology is homeomorphic to $S^1$ as a subspace of $\mathbb{R}^2$ with usual topology.
\end{corollary}

As shown in Example \ref{noHausdorff}, $\mathbb{R}$ with usual topology may not have Hausdorff one point $\mathcal{I}$-compactification for some ideals, one of such one point $\mathcal{I}$-compactification is as follows:

\begin{example} One point $\mathcal{I}$-compactification of $\mathbb{R}$ with usual topology $\mathcal{U}$ is a circle $S^1$ with topology $\tau_{S^1}$ consisting of open subset of $S^1-\{\alpha\}$ considered as a subspace of $\mathbb{R}^2$ and cofinite subset of $S^1$ containing $\alpha$, where $\alpha$ is the point at infinity of $\mathbb{R}$.
	
	Define a mapping $e: \mathbb{R}\rightarrow S^1-\{\alpha\}$ by
	$$e(x)=\left\{
	\begin{array}{ll}
		(x,\sqrt{1-x^2}) & \hbox{,~if~$|x|\leq1;$} \\
		(\frac{2x}{x^2+1},\frac{x^2-1}{x^2+1}) & \hbox{,~if~$|x|>1$}
	\end{array}
	\right.$$
	Then, $e$ is a bijection.
	Claim that, $S^1$ with $\tau_{S^1}$ is an one point $\mathcal{I}$-compactification of $\mathbb{R}$. Let us consider an $\mathcal{I}$-nonthin sequence $(x_n)_{n\in M}$ in $S^1$ which has no $\mathcal{I}$-nonthin subsequence $(x_n)_{n\in L}$ that $\mathcal{I}/_L$-convergent in $S^1-\{\alpha\}$, then $(x_n)_{n\in M}$ $\mathcal{I}/_M$-converges to $\alpha$. Otherwise, there exists an open set $U\cup\{\alpha\}$ containing $\alpha$ such that $L=\{n\in M; x_n\notin U\cup\{\alpha\}\}\notin\mathcal{I}/_M$. Then $(x_n)_{n\in L}$ is an $\mathcal{I}$-nonthin sequence in $S^1-\{\alpha\}$. Also since $S^1-(U\cup\{\alpha\})$ is finite say, $\{a_1,a_2,...,a_m\}$ that is, $x_n\in \{a_1,a_2,...,a_m\}, n\in L$. This implies, $(x_n)_{n\in L}$ has an $\mathcal{I}$-nonthin subsequence $(x_n)_{n\in K}$ that $\mathcal{I}/_K$-converges to one of such $a_i$, which contradicts our assumption. Hence, $S^1$ with $\tau_{S^1}$ is $\mathcal{I}$-compact. Also, $e: \mathbb{R}\rightarrow S^1$ is an embedding, since $e: \mathbb{R}\rightarrow e(\mathbb{R})$ is a homeomorphism and $e(\mathbb{R})=S^1-\{\alpha\}$ is a dense subspace of $S^1$. Hence, one point $\mathcal{I}$-compactification of $\mathbb{R}$ with usual topology is $S^1$ with $\tau_{S^1}$ which is not Hausdorff.

\end{example}

%%% ENTER REFERENCES IN THE FORM

\end{document}